\documentclass[12pt,a4paper]{article}

\usepackage{latexsym}
\usepackage{amsfonts,amsmath,amsthm}
\usepackage{amsthm}
\usepackage{amsmath}
\usepackage{amssymb}
\usepackage{mathrsfs}
\usepackage[top=1.5cm, bottom=1.8cm, left=2.5cm, right=2.5cm]{geometry}
\usepackage{blkarray}
\usepackage[all]{xy}

\newenvironment{changemargin}[2]{\begin{list}{}{%
\setlength{\topsep}{0pt}%
\setlength{\leftmargin}{0pt}%
\setlength{\rightmargin}{0pt}%
\setlength{\listparindent}{\parindent}%
\setlength{\itemindent}{\parindent}%
\setlength{\parsep}{0pt plus 1pt}%
\addtolength{\leftmargin}{#1}%
\addtolength{\rightmargin}{#2}%
}\item }{\end{list}}

\newtheorem{theorem}{Theorem}[section]
\newtheorem{proposition}{Proposition}[section]
\newtheorem{corollary}{Corollary}[section]
\newtheorem{definition}{Definition}[section]
\newtheorem{lemma}{Lemma}[section]

\theoremstyle{remark}
\newtheorem{remark}{Remark}

\newcommand{\fleche}[4]{                     
            \begin{array}{rcl} #1 & \rightarrow & #2 \\   %
                         #3 &\mapsto & #4          %
            \end{array}}

\newcommand{\foncIso}[5]{                     
            \begin{array}{crll}#1 :& #2 & \overset{\sim}{\rightarrow} & #3 \\   %
                         &#4 &\mapsto & #5          %
            \end{array}}

\let\svthefootnote\thefootnote

\DeclareMathOperator{\End}{End}
\DeclareMathOperator{\SLF}{SLF}
\DeclareMathOperator{\Hom}{Hom}
\DeclareMathOperator{\vect}{vect}
\DeclareMathOperator{\Ob}{Ob}
\DeclareMathOperator{\Proj}{Proj}
\DeclareMathOperator{\Top}{Top}
\DeclareMathOperator{\Soc}{Soc}

\title{\bf A note on symmetric linear forms and traces \\on the restricted quantum group $\bar U_q(\mathfrak{sl}(2))$}
\author{Matthieu \textsc{Faitg}}
\date{}

\begin{document}

\maketitle

\vspace{2em}
\begin{changemargin}{1.5cm}{1.5cm}
{\small
\noindent \textsc{Abstract}. \hspace{2pt} In this paper we prove two results about $\SLF(\bar U_q)$, the algebra of symmetric linear forms on the restricted quantum group $\bar U_q = \bar U_q\left(\mathfrak{sl}(2)\right)$. First, we express any trace on finite dimensional projective $\bar U_q$-modules as a linear combination in the basis of $\SLF(\bar U_q)$ constructed by Gainutdinov - Tipunin and also by Arike. In particular, this allows us to determine the symmetric linear form corresponding to the modified trace on projective $\bar U_q$-modules. Second, we give the explicit multiplication rules between symmetric linear forms in this basis.
}
\end{changemargin}
\vspace{0.5em}
\section{Introduction}

\indent Let 
\let\thefootnote\relax\footnote{2010 Mathematics Subject Classification. Primary 16T20 ; Secondary 17B37, 16T05.}
\addtocounter{footnote}{-1}\let\thefootnote\svthefootnote 
$\!\! \bar U_q = \bar U_q(\mathfrak{sl}(2))$ be the restricted quantum group associated to $\mathfrak{sl}(2)$ and $\SLF(\bar U_q)$ its space of {\em symmetric linear forms}, which is naturally endowed with an algebra structure. In \cite{GT} and \cite{arike}, an interesting basis of $\SLF(\bar U_q)$ is introduced, that will be called the GTA basis in the sequel, and whose construction is based on the simple and the projective $\bar U_q$-modules (see section \ref{sectionSLF}). In this paper, we prove two results about this basis, namely the relation with traces on projectives modules, and the formulas for multiplication of symmetric linear forms.

\smallskip

\indent First, we show in the general setting of a finite dimensional $k$-algebra $A$ that there is a correspondence between traces on finite dimensional projective $A$-modules and symmetric linear forms on $A$ (Theorem \ref{corTracesSLF}). In the case of $A = \bar U_q$, the natural question is to express the image of a trace through this correspondence in the GTA basis. We answer this question and show that this basis is relevant with regard to this correspondence in Theorem \ref{MainResult}. The modified trace computed in \cite{BBGe} is an interesting example of a trace on projective $\bar U_q$-modules. We determine the symmetric linear form corresponding to the modified trace, and get that it is $\mu(K^{p+1}\cdot)$, where $\mu$ is a suitably normalized right integral of $\bar U_q$ (see section \ref{ModTrace}). This last result has been found simultaneously in \cite{BBG} in a general framework including $\bar U_q$.

\smallskip

\indent With regard to the structure of algebra on $\SLF(\bar U_q)$, a natural and important problem is to determine the multiplication rules of the elements in the GTA basis. In section \ref{multiplication}, we find the decomposition of the product of two basis elements in the GTA basis. The resulting formulas are surprisingly simple (Theorem \ref{ProduitArike}). Note that a similar problem (namely the multiplication in the space of $q$-characters $\mathrm{qCh}(\bar U_q)$, which is isomorphic as an algebra to $\mathrm{SLF}(\bar U_q)$) has been solved in \cite{GT}, but I was not aware of the existence of this paper when preparing this work. It turns out that our proofs are different. In \cite{GT}, they use the fact that the multiplication in the canonical basis of $\mathcal{Z}(\bar U_q)$ is very simple. They first express the image of their basis of $\mathrm{qCh}(\bar U_q)$ through the Radford mapping in the canonical basis of $\mathcal{Z}(\bar U_q)$. This gives a basis of $\mathcal{Z}(\bar U_q)$ called the Radford basis. Then they use the $\mathcal{S}$-transformation of the $\mathrm{SL}_2(\mathbb{Z})$ representation on $\mathcal{Z}(\bar U_q)$ to express the Drinfeld basis (which is the image of their basis of $\mathrm{qCh}(\bar U_q)$ by the Drinfeld map) in the Radford basis. This gives the multiplication rules in the Drinfeld basis. Since the Drinfeld map is an isomorphism of algebras between $\mathrm{qCh}(\bar U_q)$ and $\mathcal{Z}(\bar U_q)$, this gives also the multiplication rules in the GTA basis. Here we directly work in $\mathrm{SLF}(\bar U_q)$. We first prove an elementary lemma which shows that there are not many coefficients to determine, and then we compute these coefficients by using the evaluation on suitable elements of $\bar U_q$.

\smallskip

\indent To make the paper self-contained and fix notations, we recall some facts about the structure of $\bar U_q$ and its representation theory in section \ref{preliminaries}. In section \ref{sectionSLF}, we introduce $\SLF(\bar U_q)$ and the GTA basis. We then state some properties that are needed to prove our results.

\smallskip

\indent In \cite{F}, the GTA basis and its multiplication rules are extensively used to describe in detail the projective representation of $\text{SL}_2(\mathbb{Z})$ (the mapping class group of the torus) on $\SLF(\bar U_q)$ provided by the graph algebra of the torus with the gauge algebra $\bar U_q$ (which is a quantum analogue of the algebra of functions associated to lattice gauge theory on the torus).

\medskip

\noindent {\em Acknowledgments.}\hspace{2pt} I am grateful to my advisors, St\'ephane Baseilhac and Philippe Roche, for their regular support and their useful remarks. I also thank Azat Gainutdinov for several comments about the first version of this paper and the referee for reading the manuscript carefully and for pointing out an insufficient argument in the proof of Theorem \ref{ProduitArike}.

\medskip

\noindent {\em Notations.} \hspace{2pt} If $A$ is a $k$-algebra (with $k$ a field), $V$ is a finite dimensional $A$-module and $x \in A$, we denote by $\overset{V}{x} \in \End(V)$ the representation of $x$ on the module $V$. We will work only with finite dimensional modules and mainly with left modules, thus often we simply write ``module'' instead of ``finite dimensional left module''. The {\em socle} of $V$, denoted by $\Soc(V)$ is the largest semi-simple submodule of $V$. The {\em top} of $V$, denoted by $\Top(V)$, is $V/\text{Rad}(V)$, where $\text{Rad}(V)$ is the Jacobson radical of $V$. See \cite[Chap. IV and VIII]{CR} for background material about representation theory. 
\\
\indent For $q \in \mathbb{C}\setminus \{-1,0,1\}$, we define the $q$-integer $[n]$ (with $n \in \mathbb{Z})$ and the $q$-factorial $[m]!$ (with $m \in \mathbb{N})$ by:
$$ [n] = \frac{q^n - q^{-n}}{q-q^{-1}}, \:\:\:\:\: [0]!=1, \:\: [m]! = [1][2] \ldots [m] \: \text{ for } m \geq 1. $$
In what follows $q$ is a primitive $2p$-root of unity (where $p$ is a fixed integer $\geq 2$), say $q = e^{i\pi/p}$. Observe that in this case $[n] = \frac{\sin(n\pi/p)}{\sin(\pi/p)}$, $[p]=0$ and $[p-n] = [n]$.
\\
\indent As usual, $\delta_{i,j}$ will denote the Kronecker symbol and $I_n$ the identity matrix of size $n$.

\section{Preliminaries}\label{preliminaries}
\subsection{The restricted quantum group $\bar U_q(\mathfrak{sl}(2))$}
\indent As mentioned above, $q$ is a primitive root of unity of order $2p$, with $p \geq 2$. Recall that $\bar U_q(\mathfrak{sl}(2))$, the {\em restricted quantum group} associated to $\mathfrak{sl}(2)$, is the $\mathbb{C}$-algebra generated by $E, F, K$ together with the relations 
\begin{equation*}
E^p=F^p=0, \:\:\: K^{2p}=1,\:\:\: KE=q^2EK,\:\:\: KF=q^{-2}FK, \:\:\: EF = FE + \frac{K-K^{-1}}{q-q^{-1}}.
\end{equation*}
It will be simply denoted by $\bar U_q$ in the sequel. It is a $2p^3$-dimensional Hopf algebra, with comultiplication $\Delta$, counit $\varepsilon$ and antipode $S$ given by the following formulas:
\begin{equation*}
\begin{array}{lll}
\Delta(E) = 1 \otimes E + E \otimes K, & \Delta(F) = F \otimes 1 + K^{-1} \otimes F, & \Delta(K) = K \otimes K, \\
\varepsilon(E) = 0, & \varepsilon(F) = 0, & \varepsilon(K) = 1,\\
S(E) = -EK^{-1}, & S(F) = -KF, & S(K) = K^{-1}.
\end{array}
\end{equation*}
\indent The monomials $E^mF^nK^{l}$ with $0 \leq m,n \leq p-1, \:0 \leq l \leq 2p-1$, form a basis of $\bar U_q$, usually referred as the PBW-basis. Recall the formula (see \cite[Prop. VII.1.3]{kassel}):
\begin{equation}\label{coproduitMonome}
\Delta(E^mF^nK^{l}) = \sum_{i=0}^m\sum_{j=0}^n q^{i(m-i) + j(n-j) - 2(m-i)(n-j)}{m \brack i}{n \brack j}E^{m-i}F^jK^{l+j-n} \otimes E^iF^{n-j}K^{l+m-i}.
\end{equation}
Recall that the $q$-binomial coefficients are defined by ${a \brack b} = \frac{[a]!}{[b]! [a-b]!}$ for $a \geq b$. \\
\indent Since $K$ is annihilated by the polynomial $X^{2p}-1$, which has simple roots over $\mathbb{C}$, the action of $K$ is diagonalizable on each $\bar U_q$-module, and the eigenvalues are $2p$-roots of unity.
\smallskip\\
\indent Due to the Hopf algebra structure on $\bar U_q$, its category of modules is a monoidal category with duals. It is not braided (see \cite{KS}).

\subsection{Simple and projective $\bar U_q$-modules}\label{sectionSimpleProj}
\indent The finite dimensional representations of $\bar U_q$ are classified (\cite{suter} and \cite{GSTF}). Two types of modules are important for our purposes: the simple and the projective modules. As in \cite{FGST} (see also \cite{ibanez}), we denote the simple modules by $\mathcal{X}^{\alpha}(s)$, with $\alpha \in \{\pm\}, 1 \leq s \leq p$. The modules $\mathcal{X}^{\pm}(p)$ are simple and projective simultaneously. The other indecomposable projective modules are not simple. We denote them by $\mathcal{P}^{\alpha}(s)$ with $\alpha \in \{\pm\}, 1 \leq s \leq p-1$. 
\smallskip\\
\indent The module $\mathcal{X}^{\alpha}(s)$ admits a {\em canonical basis} $\left(v_i\right)_{0 \leq i \leq s-1}$ such that
\begin{equation}\label{BaseSimple}
Kv_i = \alpha q^{s-1-2i}v_i,\: Ev_0=0,\: Ev_i=\alpha[i][s-i]v_{i-1},\: Fv_i=v_{i+1},\: Fv_{s-1}=0.
\end{equation}
The module $\mathcal{P}^{\alpha}(s)$ admits a {\em standard basis} $\left(b_i, x_j, y_k, a_{l}\right)_{\substack{0 \leq i,l \leq s-1 \\ 0 \leq j,k \leq p-s-1}}$ such that

\begin{equation}\label{BaseProjectif}
\begin{array}{lll}
Kb_i=\alpha q^{s-1-2i}b_i, & Eb_i = \alpha[i][s-i]b_{i-1} + a_{i-1}, & Fb_i = b_{i+1}, \\
 & Eb_0 = x_{p-s-1}, & Fb_{s-1}=y_0,\\
Kx_j = -\alpha q^{p-s-1-2j}x_j, & Ex_j = -\alpha[j][p-s-j]x_{j-1}, &  Fx_j = x_{j+1},  \\
 & Ex_0=0, & Fx_{p-s-1}=a_0, \\
Ky_k = -\alpha q^{p-s-1-2k}y_k, & Ey_k = -\alpha[k][p-s-k]y_{k-1}, & Fy_k = y_{k+1}, \\
 & Ey_0 = a_{s-1}, & Fy_{p-s-1} = 0, \\
Ka_{l} = \alpha q^{s-1-2l}a_{l}, & Ea_{l} = \alpha[l][s-l]a_{l-1}, & Fa_{l} = a_{l+1}, \\
 & Ea_0 = 0, & Fa_{s-1}=0.
\end{array}
\end{equation}

Note that such a basis is not unique up to scalar since we can replace $b_i$ by $b_i + \lambda a_i$ (with $\lambda \in \mathbb{C})$ without changing the action. 
\\\indent In terms of composition factors, the structure of $\mathcal{P}^{\alpha}(s)$ can be schematically represented as follows (with the basis vectors corresponding to each factor and the action of $E$ and $F$):
\begin{equation}\label{figureProj}
\xymatrix{
 & \Top\left(\mathcal{P}^{\alpha}(s)\right) \cong \mathcal{X}^{\alpha}(s), (b_i)_{0 \leq i \leq s-1} \ar[ld]^E \ar[rd]^F \ar[dd]^E& \\
(x_j)_{0 \leq j \leq p-s-1}, \mathcal{X}^{-\alpha}(p-s) \!\!\!\!\!\!\!\!\! \ar[rd]^F & & \!\!\!\!\!\!\!\!\! \mathcal{X}^{-\alpha}(p-s) \ar[ld]^E, (y_k)_{0 \leq k \leq p-s-1}\\  
 & \Soc\left(\mathcal{P}^{\alpha}(s)\right) \cong \mathcal{X}^{\alpha}(s), (a_{l})_{0 \leq l \leq s-1} &
}
\end{equation}
If we need to emphasize the module in which we are working, we will use the following notations: $v_i^{\alpha}(s)$ for the canonical basis of $\mathcal{X}^{\alpha}(s)$ and $b_i^{\alpha}(s)$, $x_j^{\alpha}(s)$, $y_k^{\alpha}(s)$, $a_{l}^{\alpha}(s)$ for a  standard basis of $\mathcal{P}^{\alpha}(s)$ (these are the notations used in \cite{arike}).
\smallskip\\
\indent Let us recall the $\bar U_q$-morphisms between these modules. Observe that $\mathcal{X}^{\alpha}(s)$ is $\bar U_q$-generated by $v^{\alpha}_0(s)$ and $\mathcal{P}^{\alpha}(s)$ is $\bar U_q$-generated by $b^{\alpha}_0(s)$, so the images of these vectors suffice to define $\bar U_q$-morphisms. $\mathcal{X}^{\alpha}(s)$ is simple, so by Schur's lemma $\End_{\bar U_q}\left(\mathcal{X}^{\alpha}(s)\right) = \mathbb{C}\text{Id}$. Since 
$$\mathcal{X}^{\alpha}(s) \cong \Top\left(\mathcal{P}^{\alpha}(s)\right) \cong \Soc\left(\mathcal{P}^{\alpha}(s)\right)$$
there exist injection and projection maps defined by:
\begin{equation*}
\begin{array}{lcl}
\mathcal{X}^{\alpha}(s) & \hookrightarrow & \mathcal{P}^{\alpha}(s)\\
v_0^{\alpha}(s) & \mapsto & a_0^{\alpha}(s)
\end{array} \:\:\: \text{ and } \:\:\:\:\,
\begin{array}{lcl}
\mathcal{P}^{\alpha}(s) & \twoheadrightarrow & \mathcal{X}^{\alpha}(s) \\
b_0^{\alpha}(s) & \mapsto & v_0^{\alpha}(s).
\end{array}
\end{equation*}
We have $\End_{\bar U_q}\left(\mathcal{P}^{\alpha}(s)\right) = \mathbb{C}\text{Id} \oplus \mathbb{C}p^{\alpha}_s$ and $\Hom_{\bar U_q}\left(\mathcal{P}^{\alpha}(s), \mathcal{P}^{-\alpha}(p-s)\right) = \mathbb{C}P^{\alpha}_s \oplus \mathbb{C}\overline{P}^{\alpha}_s$, where:
\begin{equation}\label{morphismesP}
p^{\alpha}_s\left(b_0^{\alpha}(s)\right) = a_0^{\alpha}(s), \:\:\:\:\:\: P^{\alpha}_s\left(b_0^{\alpha}(s)\right) = x_0^{-\alpha}(p-s), \:\:\:\:\:\: \overline{P}^{\alpha}_s\left(b_0^{\alpha}(s)\right) = y_0^{-\alpha}(p-s).
\end{equation}
The other Hom-spaces involving only simple modules and indecomposable projective modules are null.

\subsection{Structure of the bimodule $_{\bar U_q}\!\left(\bar U_q\right)_{\bar U_q}$ and the center of $\bar U_q$}\label{bimodEtCentre}
\indent Recall that if $M$ is a left module (over any $k$-algebra $A$), then $M^* = \Hom_{\mathbb{C}}(M, k)$ is endowed with a {\em right} $A$-module structure, given by:
\begin{equation*}
\forall \, a \in A, \: \forall\, \varphi \in M^*, \:\: \varphi a = \varphi(a \cdot)
\end{equation*}
where $\cdot$ is the place of the variable. We denote by $R^*(M)$ the so-defined right module. Note that if we define $R^*(f)$ as the transpose of $f$, then $R^*$ becomes a contravariant functor. If $A$ is a Hopf algebra, one must be aware not to confuse $R^*(M)$ with the categorical dual $M^*$, which is a left module on which $A$ acts by:
$$ \forall \, a \in A, \: \forall\, \varphi \in M^*, \:\: a\varphi = \varphi(S(a) \cdot). $$
\begin{lemma}The right $\bar U_q$-module $R^*(\mathcal{X}^{\alpha}(s))$ admits a basis $\left(\bar v_i\right)_{0 \leq i \leq s-1}$ such that
\begin{equation*}
\begin{array}{lllll}
\bar v_i K = \alpha q^{1-s+2i}\bar v_i, & \bar v_i E = \alpha[i][s-i] \bar v_{i-1}, & \bar v_0 E = 0, & v_i F = \bar v_{i+1}, & \bar v_{s-1}F = 0. \\
\end{array}
\end{equation*}
The right $\bar U_q$-module $R^*(\mathcal{P}^{\alpha}(s))$ admits a basis $\left(\bar b_{i}, \bar x_j, \bar y_k, \bar a_l\right)_{\substack{0 \leq i,l \leq s-1 \\ 0 \leq j ,k\leq p-s-1}}$ such that
\begin{equation*}
\begin{array}{lll}
   \bar b_i K = \alpha q^{1-s+2i}\bar b_i, & \bar b_i E = \bar a_{i-1} + \alpha[i][s-i] \bar b_{i-1}, & \bar b_i F = \bar b_{i+1}, \\
      & \bar b_0 E = \bar x_{p-s-1}, & \bar b_{s-1}F = \bar y_0, \\
   \bar x_j K = -\alpha q^{-p+s+1+2j} \bar x_j, & \bar x_j E = -\alpha [j][p-s-j]\bar x_{j-1}, & \bar x_j F = \bar x_{j+1}, \\
      &  \bar x_{0} E = 0, & \bar x_{p-s-1} F = \bar a_0, \\
   \bar y_{k} K = -\alpha q^{-p+s+1+2k} \bar y_k, & \bar y_k E = -\alpha[k][p-s-k] \bar y_{k-1}, & \bar y_k F = \bar y_{k+1}, \\
      &  \bar y_0 E = \bar a_{s-1}, & \bar y_{p-s-1}F = 0, \\
   \bar a_{l} K = \alpha q^{1-s+2l}\bar a_{l}, & \bar a_{l} E = \alpha[l][s-l] \bar a_{l-1}, & \bar a_{l} F = \bar a_{l+1}, \\
      & \bar a_0 E = 0, & \bar a_{s-1}F = 0.
\end{array}
\end{equation*}
\end{lemma}
\noindent Such basis will be termed respectively a {\em canonical basis} and a {\em standard basis} in the sequel.
\begin{proof}
Let $(v^i)_{0 \leq i \leq s-1}$ be the basis dual to the canonical basis given in (\ref{BaseSimple}). Then $\bar v_i = v^{s-1-i}$ gives the desired result. Similarly, let $\left(b^i, x^j, y^k, a^{l}\right)_{\substack{0 \leq i,l \leq s-1 \\ 0 \leq j,k \leq p-s-1}}$ be the basis dual to a standard basis given in (\ref{BaseProjectif}). Then 
$$ \bar b_i = a^{s-1-i}, \:\: \bar x_j = y^{p-s-1-j}, \:\: \bar y_k = x^{p-s-1-k}, \:\: \bar a_{l} = b^{s-1-l} $$
gives the desired result.
\end{proof}

\indent We denote by $_{\bar U_q}\!\left(\bar U_q\right)_{\bar U_q}$ the regular bimodule, where the left and right actions are respectively the left and right multiplication of $\bar U_q$ on itself. Recall that a block of $_{\bar U_q}\!\left(\bar U_q\right)_{\bar U_q}$ is just an indecomposable two-sided ideal (see \cite[Section 55]{CR}). The block decomposition of $\bar U_q$ is (see \cite{FGST})
\begin{equation*}
_{\bar U_q}\!\left(\bar U_q\right)_{\bar U_q} = \bigoplus_{s=0}^{p} Q(s)
\end{equation*}
where the structure of each block $Q(s)$ as a left $\bar U_q$-module is: 
\begin{equation}\label{leftBlock}
\begin{array}{l}
Q(0) \cong p\mathcal{X}^{-}(p), \:\:\:\:\: Q(p) \cong p\mathcal{X}^{+}(p),\\
Q(s) \cong s\mathcal{P}^{+}(s) \oplus (p-s)\mathcal{P}^{-}(p-s) \: \text{ for } 1 \leq s \leq p-1
\end{array}
\end{equation}
and the structure of each block as a right $\bar U_q$-module is:
\begin{equation*}
\begin{array}{l}
Q(0) \cong pR^*\!\left(\mathcal{X}^{-}(p)\right), \:\:\:\:\: Q(p) \cong pR^*\!\left(\mathcal{X}^{+}(p)\right),\\
Q(s) \cong sR^*\!\left(\mathcal{P}^{+}(s)\right) \oplus (p-s)R^*\!\left(\mathcal{P}^{-}(p-s)\right) \: \text{ for } 1 \leq s \leq p-1.
\end{array}
\end{equation*}

The following proposition is a reformulation of \cite[Prop. 4.4.2]{FGST} (see also \cite[Th. II.1.4]{ibanez}). It will be used for the proof of Theorem \ref{MainResult}.
\begin{proposition}\label{baseBloc}
For $1 \leq s \leq p-1$, the block $Q(s)$ admits a basis $$\left(B^{++}_{ab}(s), X^{-+}_{cd}(s), Y^{-+}_{ef}(s), A^{++}_{gh}(s), B^{--}_{ij}(s), X^{+-}_{kl}(s), Y^{+-}_{mn}(s), A^{--}_{or}(s) \right)$$
with $0 \leq a,b,d,f,g,h,k,m \leq s-1, \:\: 0 \leq c, e, i,j,l,n,o,r \leq p-s-1$, such that 
\begin{enumerate}
\item $\forall\, 0 \leq j \leq s-1, \:\: \left(B^{++}_{ij}(s), X^{-+}_{kj}(s), Y^{-+}_{lj}(s), A^{++}_{mj}(s)\right)_{\substack{0 \leq i,m \leq s-1 \\ 0 \leq k,l \leq p-s-1}}$ is a standard basis of $\mathcal{P}^{+}(s)$ for the left action.
\item $\forall \, 0 \leq j \leq p-s-1, \:\: \left(B^{--}_{ij}(s), X^{+-}_{kj}(s), Y^{+-}_{lj}(s), A^{--}_{mj}(s)\right)_{\substack{0 \leq k,l \leq s-1 \\ 0 \leq i,m \leq p-s-1}}$ is a standard basis of $\mathcal{P}^{-}(p-s)$ for the left action.
\item $\forall\, 0 \leq i \leq s-1, \:\: \left(B^{++}_{ij}(s), X^{+-}_{ik}(s), Y^{+-}_{il}(s), A^{++}_{im}(s)\right)_{\substack{0 \leq j,m \leq s-1 \\ 0 \leq k,l \leq p-s-1}}$ is a standard basis of $R^*\left(\mathcal{P}^{+}(s)\right)$ for the right action.
\item $\forall\, 0 \leq i \leq p-s-1, \:\: \left(B^{--}_{ij}(s), X^{-+}_{ik}(s), Y^{-+}_{il}(s), A^{--}_{im}(s)\right)_{\substack{0 \leq k,l \leq s-1 \\ 0 \leq j,m \leq p-s-1}}$ is a standard basis of $R^*\left(\mathcal{P}^{-}(p-s)\right)$ for the right action.
\end{enumerate}
The block $Q(0)$ admits a basis $\left(A^{--}_{ij}(0)\right)_{0 \leq i,j \leq p-1}$ such that
\begin{enumerate}
\item $\forall \, 0 \leq j \leq p-1, \:\: \left(A^{--}_{ij}(0)\right)_{0 \leq i \leq p-1}$ is a standard basis of $\mathcal{X}^{-}(p)$ for the left action.
\item $\forall \, 0 \leq i \leq p-1, \:\: \left(A^{--}_{ij}(0)\right)_{0 \leq j \leq p-1}$ is a standard basis of $R^*\left(\mathcal{X}^{-}(p)\right)$ for the right action.
\end{enumerate}
The block $Q(p)$ admits a basis $\left(A^{++}_{ij}(p)\right)_{0 \leq i,j \leq p-1}$ such that
\begin{enumerate}
\item $\forall \, 0 \leq j \leq p-1, \:\: \left(A^{++}_{ij}(p)\right)_{0 \leq i \leq p-1}$ is a standard basis of $\mathcal{X}^{+}(p)$ for the left action.
\item $\forall \, 0 \leq i \leq p-1, \:\: \left(A^{++}_{ij}(p)\right)_{0 \leq j \leq p-1}$ is a standard basis of $R^*\left(\mathcal{X}^{+}(p)\right)$ for the right action.
\end{enumerate}
\end{proposition}

\noindent As in \cite{FGST}, the structure of $Q(s)$ in terms of composition factors can be schematically represented as follows (each vertex represents a composition factor and is labelled by the basis vectors of this factor):
{\footnotesize
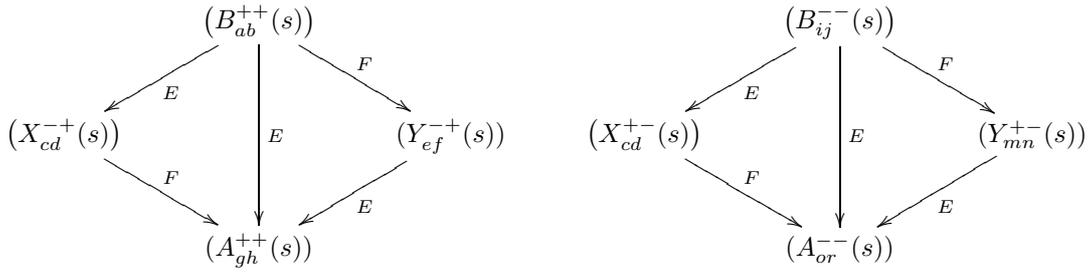
$$
\xymatrix{
        & \left(B^{++}_{ab}(s)\right) \ar[ld]^E \ar[rd]^F \ar[dd]^E &      &        &  \left(B^{--}_{ij}(s)\right) \ar[ld]^E \ar[rd]^F \ar[dd]^E&    \\
\left(X^{-+}_{cd}(s)\right) \ar[dr]^F &    &  (Y^{-+}_{ef}(s)) \ar[dl]^E  &  \left(X^{+-}_{cd}(s)\right) \ar[dr]^F  &  &  \left(Y^{+-}_{mn}(s) \ar[dl]^E\right) \\
        &  (A^{++}_{gh}(s)) &      &        &  \left(A^{--}_{or}(s)\right) & 
}
$$
}
for the left action, and
{\footnotesize
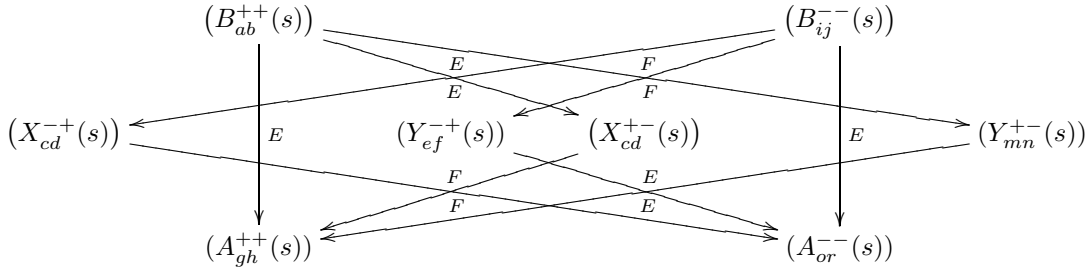
$$
\xymatrix{
        & \left(B^{++}_{ab}(s)\right) \ar[rrd]^E \ar[rrrrd]^F \ar[dd]^E&      &        &  \left(B^{--}_{ij}(s)\right) \ar[lllld]^E \ar[lld]^F \ar[dd]^E&    \\
\left(X^{-+}_{cd}(s)\right) \ar[rrrrd]^F &    &  (Y^{-+}_{ef}(s)) \ar[rrd]^E  &  \left(X^{+-}_{cd}(s)\right) \ar[lld]^F &  &  \left(Y^{+-}_{mn}(s)\right) \ar[lllld]^E\\
        &  (A^{++}_{gh}(s)) &      &        &  \left(A^{--}_{or}(s)\right) & 
}
$$
}
for the right action.
\smallskip\\
\indent The knowledge of the structure of the bimodule $_{\bar U_q}\!\left(\bar U_q\right)_{\bar U_q}$ allows us to determine the center of $\bar U_q$. Indeed, each central element determines a bimodule endomorphism and conversely. Recall from \cite{FGST} that $\mathcal{Z}(\bar U_q)$ is a $(3p-1)$-dimensional algebra with basis elements $e_s \:\:(0 \leq s \leq p)$ and $w^{\pm}_t \:\: (1 \leq t \leq p-1)$. The element $e_s$ is just the unit of the block $Q(s)$, thus by (\ref{leftBlock}) and (\ref{figureProj}) the action of $e_s$ on the simple and the projective modules is given by
\begin{equation}\label{actionEs}
\begin{array}{llll}
\text{For } s=0, \:\: &e_0 v_0^+(t) = 0, & e_0 v_0^-(t) =\delta_{t,p}v_0^-(p), & e_0 b_0^{\pm}(t) = 0, \\
\text{For }1 \leq s \leq p-1, \:\: &e_s v_0^{+}(t) = \delta_{s,t}v_0^+(s), & e_s v_0^-(t) = \delta_{p-s,t}v_0^{-}(p-s), & \\
 & e_s b_0^+(t) = \delta_{s,t}b_0^+(s), & e_s b_0^-(t) = \delta_{t,p-s}b_0^-(p-s), & \\
\text{For } s=p, \:\: &e_p v_0^+(t) = \delta_{t,p}v_0^+(p), & e_p v_0^-(t) =0, & e_p b_0^{\pm}(t) = 0
\end{array}
\end{equation}
while for the elements $w^{\pm}_s$:
\begin{equation}\label{actionWs}
\begin{array}{lll}
w^+_s v_0^{\pm}(t) = 0, & w^+_s b_0^+(t) = \delta_{s,t}a_0^+(s), & w^+_s b_0^-(t) = 0,\\
w^-_s v_0^{\pm}(t) = 0, & w^-_s b_0^+(t) = 0, & w^-_s b_0^-(t) = \delta_{t,p-s}a_0^-(p-s).\\
\end{array}
\end{equation}
Observe that 
$$\overset{\mathcal{P}^+(s)}{w_s^+} = p^+_s, \:\:\:\: \overset{\mathcal{P}^-(p-s)}{w_s^-} = p^-_{p-s}.$$
The action of the central elements on $\mathcal{P}^{\alpha}(s)$ is enough to recover their action on every module, using projective covers. From these formulas, we deduce the multiplication rules of these elements:
\begin{equation}\label{produitCentre}
e_se_t = \delta_{s,t}e_s, \:\:\: e_sw^{\pm}_t = \delta_{s,t}w^{\pm}_s, \:\:\: w^{\pm}_sw^{\pm}_t = 0.
\end{equation}
\noindent Let us mention that the idempotents $e_s$ are not primitive: there exists primitive orthogonal idempotents $e_{s,i}$ such that $e_s = \sum_{i}e_{s,i}$, see \cite{arike}.

\section{Symmetric linear forms and the GTA basis}\label{sectionSLF}
\indent Let $A$ be a $k$-algebra, and let $\SLF(A)$ be the space of {\em symmetric linear forms} on $A$:
$$ \SLF(A) = \left\{\varphi \in A^* \, \vert \, \forall\, x,y \in A, \:\: \varphi(xy) = \varphi(yx)\right\}. $$
If $A$ is a bialgebra, then $A^*$ is an algebra whose product is defined by:
$$ \varphi\psi(x) = \sum_{(x)}\varphi(x')\psi(x'') $$
with $\Delta(x) = \sum_{(x)} x' \otimes x''$ (Sweedler's notation, see e.g. \cite[Chap. 3]{kassel}). Then $\SLF(A)$ is a subalgebra of $A^*$. Indeed, if $\varphi, \psi \in \SLF(A)$, we have:
$$ \varphi\psi(xy) = \sum_{(x),(y)}\varphi(x'y')\psi(x''y'') = \sum_{(x),(y)}\varphi(y'x')\psi(y''x'') = \varphi\psi(yx) $$
which shows that $\varphi\psi \in \SLF(A)$. If moreover $A$ is finite dimensional, then $A^*$ is a bialgebra whose coproduct is defined by $\Delta(\varphi)(x \otimes y) = \varphi(xy)$, but $\SLF(A)$ is not in general a sub-coalgebra of $A^*$, see Remark \ref{remarkCogebreSLF} below.
\smallskip\\
\indent Recall (see \cite{FGST}) that there is a universal $R$-matrix $R$ belonging to the extension of $\bar U_q$ by a square root of $K$. It satisfies $RR' \in \bar U_q^{\otimes 2}$, where $R' = \tau(R)$, with $\tau$ the flip map defined by $\tau(x \otimes y) = y \otimes x$. Moreover $\bar U_q$ is factorizable (in a generalized sense since it does not contain the $R$-matrix) and $K^{p+1}$ is a pivotal element, thus it is known from general theory that the Drinfeld morphism which we denote $\mathcal{D}$ provides an isomorphism of algebras
\begin{equation}\label{morphismeDrinfeld}
\foncIso{\mathcal{D}}{\SLF(\bar U_q)}{\mathcal{Z}(\bar U_q)}{\varphi}{\left(\varphi \otimes \text{Id}\right)\left((K^{p+1} \otimes 1)\cdot RR'\right)}
\end{equation}
\indent Let $A$ be a $k$-algebra, and $V$ an $n$-dimensional $A$-module. If we choose a basis on $V$, we get a matrix $\overset{V}{T} \in \text{Mat}_{n}(A^*)$, simply defined by
\begin{equation}\label{defT}
\overset{V}{T}(x) = \overset{V}{x}
\end{equation}
where $\overset{V}{x}$ is the representation of $x \in A$ in $\End(V)$ expressed in the choosen basis. In our case, we will always choose the canonical bases of the simple modules and standard bases of the projective modules.
\smallskip\\
\indent An interesting basis of $\SLF(\bar U_q)$ was found by Gainutdinov and Tipunin in \cite{GT} and by Arike in \cite{arike}. To be precise, a basis of the space $\mathrm{qCh}(\bar U_q)$ of $q$-characters is constructed in \cite{GT}, but the shift by the pivotal element $g=K^{p+1}$ provides an isomorphism 
$$\mathrm{qCh}(\bar U_q) \overset{\sim}{\rightarrow} \mathrm{SLF}(\bar U_q), \:\:\:\:\psi \mapsto\psi(g\,\cdot).$$
\indent This basis is built from the simple and the projective modules. First, define $2p$ linear forms\footnote{The correspondence of notations with \cite{arike} is: $T^+_s = \chi^+_s$, $T^-_s = \chi^-_{p-s}$. The letter $T$ is here reserved for the matrices $\overset{V}{T}$ described above.} $\chi^{\alpha}_s$, $\alpha \in \{\pm\}, 1 \leq s \leq p$, by:
\begin{equation}\label{defChis}
\chi^{\alpha}_s = \text{tr}(\overset{\mathcal{X}^{\alpha}(s)}{T}).
\end{equation}
They are obviously symmetric. Observe that $\chi^+_1 = \varepsilon$ is the unit for the algebra structure on $\SLF(\bar U_q)$ described above. To construct the $p-1$ missing linear forms, observe with the help of (\ref{figureProj}) that the matrix of the action on $\mathcal{P}^{\alpha}(s)$ has the following block form in a standard basis:
\begin{equation*}
\overset{\mathcal{P}^{\alpha}(s)}{T} =
\begin{blockarray}{ccccc}
(b_i) & (x_j) & (y_k) & (a_{l}) &\\
\begin{block}{(cccc)c}
\overset{\mathcal{X}^{\alpha}(s)}{T} & 0 & 0 & 0 & (b_i) \\
A^{\alpha}_s & \overset{\mathcal{X}^{-\alpha}(p-s)}{T} & 0 & 0 & (x_j)\\
B^{\alpha}_s & 0 & \overset{\mathcal{X}^{-\alpha}(p-s)}{T} & 0 & (y_k)\\
H^{\alpha}_s & D^{\alpha}_s & C^{\alpha}_s & \overset{\mathcal{X}^{\alpha}(s)}{T} & (a_{l}).\\
\end{block}
\end{blockarray}
\end{equation*}
It is not difficult to see that these matrices satisfy the following symmetries:
\begin{equation*}
A^-_{p-s} = C^+_s, \:\:\: B^-_{p-s} = D^+_s, \:\:\: D^-_{p-s} = B^+_s, \:\:\: C^-_{p-s} = A^+_s .
\end{equation*}
By computing the matrices $\overset{\mathcal{P}^{+}(s)}{(xy)} = \overset{\mathcal{P}^{+}(s)}{x}\overset{\mathcal{P}^{+}(s)}{y}$ and $\overset{\mathcal{P}^{-}(p-s)}{(xy)} = \overset{\mathcal{P}^{-}(p-s)}{x}\overset{\mathcal{P}^{-}(p-s)}{y}$, these symmetries allow us to see that the linear form $G_s \: (1 \leq s \leq p-1)$ defined by
\begin{equation}\label{defGs}
G_s = \text{tr}(H^+_s) + \text{tr}(H^-_{p-s})
\end{equation}
is a symmetric linear form.
\smallskip\\ \indent It is instructive for our purposes to see a proof that these symmetric linear forms are linearly independent. Let us begin by introducing important elements for $0 \leq n \leq p-1$ (they are discrete Fourier transforms of $(K^{l})_{0 \leq l \leq 2p-1}$): 
\begin{equation*}
\Phi^{\alpha}_n = \frac{1}{2p}\sum_{l=0}^{2p-1}\left(\alpha q^{-n}\right)^{l}K^{l}.
\end{equation*}
The following easy lemma shows that these elements allow one to select vectors which have a given weight, and this turns out to be very useful. 
\begin{lemma}\label{lemmePoids}
1) Let $M$ be a left $\bar U_q$-module, and let $m^+_i(s)$ be a vector of weight $q^{s-1-2i}$, $m^-_i(p-s)$ be a vector of weight $-q^{(p-s)-1-2i} = q^{-s-1-2i}$, $m^-_i(s)$ be a vector of weight $-q^{s-1-2i}$, $m^+_i(p-s)$ be a vector of weight $q^{(p-s)-1-2i} = -q^{-s-1-2i}$. Then: 
\begin{align*}
&\Phi^{+}_{s-1}m_i^+(s) = \delta_{i,0}m^+_0(s), \:\:\: \Phi^{+}_{s-1}m_i^-(p-s) = 0,\\
&\Phi^{-}_{s-1}m_i^-(s) =  \delta_{i,0}m^-_0(s), \:\:\: \Phi^{-}_{s-1}m_i^+(p-s) = 0.
\end{align*}
2) Let $N$ be a right $\bar U_q$-module, and let $n^+_i(s)$ be a vector of weight $q^{1-s+2i}$, $n^-_i(p-s)$ be a vector of weight $-q^{1-(p-s)+2i} = q^{1+s+2i}$, $n^-_i(s)$ be a vector of weight $-q^{1-s+2i}$, $n^+_i(p-s)$ be a vector of weight $q^{1-(p-s)+2i} = -q^{1+s+2i}$. Then:
\begin{align*}
&n_i^+(s)\Phi^{+}_{s-1} = \delta_{i,s-1}n^+_{s-1}(s),\:\:\:n_i^-(p-s)\Phi^{+}_{s-1} = 0,\\
&n_i^-(s)\Phi^{-}_{s-1} = \delta_{i,s-1}n^-_{s-1}(s) ,\:\:\:n_i^+(p-s)\Phi^{-}_{s-1} = 0.
\end{align*}
\end{lemma}
\begin{proof}
It follows from easy computations with sums of roots of unity.
\end{proof}

We can now state the key observation.

\begin{proposition}\label{propArike}
Let
$$\varphi = \sum_{s=1}^{p} \left(\lambda^{+}_s\chi^{+}_s + \lambda^{-}_s\chi^{-}_s\right) + \sum_{s'=1}^{p-1}\mu_{s'} G_{s'} \in \text{\em SLF}\left(\bar U_q\right).$$
Then: 
$$ \lambda^+_s = \varphi\left(\Phi^+_{s-1}e_s\right),\:\: \lambda^-_s = \varphi\left(\Phi^-_{s-1}e_{p-s}\right), \:\: \mu_{s'} = \frac{\varphi\left(w^+_{s'}\right)}{s'} = \frac{\varphi(w_{s'}^-)}{p-s'}. $$
\end{proposition}
\begin{proof}
It is a corollary of (\ref{actionEs}) and (\ref{actionWs}). Indeed, we have:
$$ 
\begin{array}{l l l}
\overset{\mathcal{X}^+(s)}{T}(e_t) = \delta_{s,t}I_s, \:\:\: & \overset{\mathcal{X}^+(s)}{T}(w^{\pm}_t) = 0, \:\:\: & \overset{\mathcal{X}^-(s)}{T}(e_t) = \delta_{s, p-t}I_s,\\
\overset{\mathcal{X}^-(s)}{T}(w^{\pm}_t) = 0,& H^{\pm}_s(e_t) = 0, & H^+_s(w^+_s) = \delta_{s,t}I_{s}, \\
H^+_s(w^-_t) = 0, & H^-_{p-s}(w^+_t) = 0, & H^-_{p-s}(w^-_t) = \delta_{s,t}I_{p-s}.
\end{array}
$$
This gives the formula for $\mu_s$. The formulas for $\lambda^{\pm}_s$ follow from this and Lemma \ref{lemmePoids}.
\end{proof}

\noindent If we have $\sum_{s=1}^{p} \left(\lambda^{+}_s\chi^{+}_s + \lambda^{-}_s\chi^{-}_s\right) + \sum_{s'=1}^{p-1}\mu_{s'} G_{s'} = 0$, we can evaluate the left-hand side on the elements appearing in Proposition \ref{propArike} to get that all the coefficients are equal to $0$. Thus we have a free family of cardinal $3p-1$, hence a basis of $\SLF(\bar U_q)$, since $\dim(\SLF(\bar U_q)) = 3p-1$ by (\ref{morphismeDrinfeld}).

\begin{theorem}\label{thAri}
The symmetric linear forms $\chi^{\pm}_s \: (1 \leq s \leq p)$ and $G_{s'} \: (1 \leq s' \leq p-1)$ form a basis of $\SLF(\bar U_q)$.
\end{theorem}

\begin{definition}
The basis of Theorem \ref{thAri} will be called the {\em GTA basis} (for Gainutdinov, Tipunin, Arike).
\end{definition}

\begin{remark}\label{remarkCogebreSLF}
Let $\varphi \in \SLF(\bar U_q)$. It is easy to see that $\varphi(K^jE^nF^m) = 0$ if $n \neq m$. From this we deduce that $\SLF(\bar U_q)$ is not a sub-coalgebra of $\bar U_q^*$. Indeed, write $\Delta(\chi^+_2) = \sum_i \varphi_i \otimes \psi_i$, and assume that $\varphi_i, \psi_i \in \SLF(\bar U_q)$. Then $ 1 = \chi^+_2(EF) = \sum_i \varphi_i(E)\psi_i(F) = 0 $, a contradiction.
\end{remark}

\begin{remark}
If we choose a basis of $\mathcal{Z}(\bar U_q)$, then its dual basis can not be entirely contained in $\SLF(\bar U_q)$. Indeed, let $\varphi = \sum_{s=0}^p\lambda^{\pm}_s \chi^{\pm}_s  + \sum_{s=1}^{p-1}\mu_s G_s \in \SLF(\bar U_q)$. Then $\varphi(w^+_s) = s\mu_s, \varphi(w^-_s) = (p-s)\mu_s$, and we see that there does not exist $\varphi \in \SLF(\bar U_q)$ such that $\varphi(w^+_s) = 1, \: \varphi(w^-_s) = 0$. Hence, $\SLF(\bar U_q) \subset \bar U_q^*$ is not the dual of $\mathcal{Z}(\bar U_q) \subset \bar U_q$.
\end{remark}

\section{Traces on projective $\bar U_q$-modules and the GTA basis}\label{sectionArike}

\subsection{Correspondence between traces and symmetric linear forms}\label{GenCor}
\indent Let $A$ be a finite dimensional $k$-algebra. We have an anti-isomorphism of algebras:
$$A \to\End_{A}(A), \:\:\: a \mapsto \rho_a \text{ defined by } \rho_a(x) = xa. $$
Observe that the right action of $A$ naturally appears.
Let $t$ be a trace on $A$, that is, an element of $\SLF(\End_A(A))$. Then:
$$t(\rho_{ab}) = t(\rho_b \circ \rho_a) = t(\rho_a \circ \rho_b) = t(\rho_{ba}).$$
So we get an isomorphism of vector spaces
\begin{equation*}
\fleche{\left\{\text{Traces on } \End_{A}(A)\right\} = \SLF\left(\End_{A}(A)\right)}{\SLF(A)}{t}{\varphi^t \text{ defined by } \varphi^t(a) = t(\rho_a).}
\end{equation*}
whose inverse is:
\begin{equation*}
\fleche{\SLF(A)}{\left\{\text{Traces on } \End_{A}(A)\right\} = \SLF\left(\End_{A}(A)\right)}{\varphi}{t^{\varphi} \text{ defined by } t^{\varphi}(\rho_a) = \varphi(a).}
\end{equation*}
In the case of $A=\bar U_q$, we can express $\varphi^t$ in the GTA basis, which will be the object of the next section.
\medskip\\
\indent Let $\Proj_A$ be the full subcategory of the category of finite dimensional $A$-modules whose objects are the projective $A$-modules.
\begin{definition}
A {\em trace} on $\Proj_A$ is a family of linear maps $t = \left(t_U : \End_A(U) \to k\right)_{U \in \Proj_A}$ such that
$$
\forall \, f \in \Hom_A(U,V), \: \forall\, g \in \Hom_A(V,U), \:\:  t_V(g \circ f) = t_U(f \circ g).
$$
We denote by $\mathcal{T}_{\Proj_A}$ the vector space of traces on $\Proj_A$.
\end{definition}
This cyclic property of traces on $\Proj_A$ is one of the axioms of the so-called modified traces, defined for instance in \cite{GKP}. Note that this definition could be restated in the following way (and could be generalized to other abelian full subcategories than $\Proj_A$).
\begin{lemma}\label{lemmeCoherence}
Let $t=(t_U : \End_A(U) \to k)_{U \in \Proj_A}$ be a family of linear maps. Then $t$ is a trace on $\text{\em Proj}_A$ if and only if:
\begin{itemize}
\item $\forall\, f,g \in \End_A(U), \:\:t_U(g \circ f) = t_U(f \circ g)$,
\item $t_{U \oplus V}(f) = t_{U}(p_U \circ f \circ i_U) + t_V(p_V \circ f \circ i_V)$, where $p_U, p_V$ are the canonical projection maps and $i_U, i_V$ are the canonical injection maps.
\end{itemize}
\end{lemma} 
\begin{proof}
If $t$ is a trace and $f \in \End_A(U \oplus V)$, we have:
$$ t_{U \oplus V}(f) =  t_{U \oplus V}((i_Up_U + i_Vp_V)f) = t_U(p_Ufi_U) + t_V(p_Vfi_V). $$
Conversely, let $f : U \to V$, $g : V \to U$. Define $F = i_V f p_U, G = i_U g p_V$. Then $FG = i_V f g p_V$ and $GF = i_U g f p_U$. We have $p_U GF i_U = gf$, $p_V GF i_V = 0$, $p_U FG i_U = 0$, $p_V FG i_V = fg$, thus:
$$ t_{V}(fg) = t_{U \oplus V}(FG) = t_{U \oplus V}(GF) = t_U(gf). $$
This shows the equivalence.
\end{proof}

\indent Now, consider:
\begin{equation*}
\begin{array}{crcccl}
\Pi_A \: : & \mathcal{T}_{\Proj_A} & \to & \SLF(\End_A(A)) & \overset{\sim}{\to} & \SLF(A)\\
&t=(t_U)_{U \in \Ob(\text{Proj}_A)} & \mapsto & t_A & \mapsto & \varphi^t \text{ defined by } \varphi^t(a) = t_A(\rho_a).
\end{array}
\end{equation*}

\begin{theorem}\label{corTracesSLF}
The map $\Pi_A$ is an isomorphism. In other words, $t_A$ entirely characterizes $t = (t_U)$.
\end{theorem}
\begin{proof}
For all the facts concerning PIMs (Principal Indecomposable Modules) and idempotents in finite dimensional $k$-algebras, we refer to \cite[Chap. VIII]{CR}. We first show that $\Pi_A$ is surjective. Let: 
$$ 1 = e_1 + \ldots + e_n $$
be a decomposition of the unit into primitive orthogonal idempotents ($e_ie_j = \delta_{i,j}e_i$). Then the PIMs of $A$ are isomorphic to the left ideals $Ae_i$ (possibly with multiplicity). We have isomorphisms of vector spaces:
$$ \Hom_A(Ae_i, Ae_j) \overset{\sim}{\longrightarrow} e_iAe_j,\:\:\:\: f \mapsto f(e_i). $$
For every $\varphi \in \SLF(A)$, define $t^{\varphi}_{Ae_i}$ by:
$$ t^{\varphi}_{Ae_i}(f) = \varphi(f(e_i)). $$
Let $f : Ae_i \to Ae_j$, $g : Ae_j \to Ae_i$, and put $f(e_i) = e_ia_fe_j$, $g(e_j) = e_ja_ge_i$. Then using the idempotence of the $e_i$'s and the symmetry of $\varphi$ we get:
$$ t^{\varphi}_{Ae_i}(g \circ f) = \varphi(g \circ f(e_i)) = \varphi\left((e_ia_fe_j)(e_ja_ge_i)\right) = \varphi\left((e_ja_ge_i)(e_ia_fe_j)\right) = \varphi(f \circ g(e_j)) =  t^{\varphi}_{Ae_j}(f \circ g). $$

We know that every projective module is isomorphic to a direct sum of PIMs, so we extend $t^{\varphi}$ to $\text{Proj}_A$ by the following formula:
$$ t_{\bigoplus_{l} A_{l}}(f) = \sum_{l}t_{Ae_{l}}(i_{l} \circ f \circ p_{l}) $$
where $p_{l}$ and $i_{l}$ are the canonical injection and projection maps. By Lemma \ref{lemmeCoherence}, this defines a trace on $\Proj_A$.
We then show that $\Pi_A(t^{\varphi}) = \varphi$, proving surjectivity:
\begin{align*}
\Pi_A(t^{\varphi})(a) &= t^{\varphi}_A(\rho_a) = \sum_{j=1}^n t^{\varphi}_{Ae_j}\left(p_{j} \circ {\rho_a} \circ i_j\right) = \sum_{j=1}^n \varphi\left(p_{j} \circ {\rho_a}(e_j)\right) = \sum_{j=1}^n \varphi\left(p_{Ae_j}(e_ja)\right) \\&= \sum_{j, k = 1}^n \varphi\left(p_{Ae_j}(e_jae_k)\right) = \sum_{j=1}^n \varphi\left(e_jae_j\right) = \sum_{j=1}^n \varphi\left(ae_j\right) = \varphi(a).
\end{align*}
Note that we used that the $e_j$'s are idempotents and that $a = \sum_{j=1}^n ae_j$. We now show injectivity. Assume that $\Pi_A(t)=0$. Then:
$$\forall \, a \in A, \:\: t_A(\rho_a) = \sum_{j=1}^n t_{Ae_j}(p_{j} \circ {\rho_a} \circ i_j) = 0. $$
Let $f : Ae_j \to Ae_j$, with $f(e_{j}) = e_{j}a_fe_{j}$. Since $\rho_{f(e_{j})}(e_{l}) = \delta_{j,l}e_{j}a_fe_{j}$, we have $p_{j} \circ \rho_{f(e_{j})} \circ i_{j} = f$ and $p_{l} \circ \rho_{f(e_{j})} \circ i_{l}=0$ if $l \neq j$. Hence:
$$ t_{Ae_{j}}(f) = t_A(\rho_{f(e_{j})}) = 0. $$
Then $t_{Ae_{j}} = 0$ for each $j$, so that $t=0$.
\end{proof}

\subsection{Link with the GTA basis}
\indent We leave the general case and focus on $A=\bar U_q$. The following theorem expresses $\Pi_{\bar U_q}$ in the GTA basis.
\begin{theorem}\label{MainResult}
Let $t=(t_U)_{U \in\Proj_{\bar U_q}}$ be a trace on $\Proj_{\bar U_q}$. Then:
\[
\Pi_{\bar U_q}(t) = t_{\mathcal{X}^+(p)}(\text{\em Id})\chi^+_p + t_{\mathcal{X}^-(p)}(\text{\em Id})\chi^-_p + \sum_{s=1}^{p-1} \left(t_{\mathcal{P}^+(s)}(\text{\em Id})\chi^+_s + t_{\mathcal{P}^-(s)}(\text{\em Id})\chi^-_s + t_{\mathcal{P}^+(s)}(p^+_s)G_s\right).
\]
\end{theorem}
\begin{proof}
First of all, we write the decomposition of the left regular representation of $\bar U_q$, assigning an index to the multiple factors: 
$$ \bar U_q = \bigoplus_{s=1}^{p-1}\left(\bigoplus_{j=0}^{s-1}\mathcal{P}_j^+(s) \oplus \mathcal{P}_j^-(s)\right) \oplus \bigoplus_{j=0}^{p-1}\mathcal{X}_j^+(p) \oplus \mathcal{X}_j^-(p). $$
Thus, since $t$ is a trace: 
\begin{align*}
t_{\bar U_q}(\rho_a) = &\sum_{s=1}^{p-1}\left(\sum_{j=0}^{s-1}t_{\mathcal{P}_j^+(s)}\left(p_{\mathcal{P}_j^+(s)} \circ \rho_a \circ i_{\mathcal{P}_j^+(s)}\right) + t_{\mathcal{P}_j^-(s)}\left(p_{\mathcal{P}_j^-(s)} \circ \rho_a \circ i_{\mathcal{P}_j^-(s)}\right)\right)\\
& + \sum_{j=0}^{p-1}t_{\mathcal{X}_j^+(p)}\left(p_{\mathcal{X}_j^+(p)} \circ \rho_a \circ i_{\mathcal{X}_j^+(p)}\right) + t_{\mathcal{X}_j^-(p)}\left(p_{\mathcal{X}_j^-(p)} \circ \rho_a \circ i_{\mathcal{X}_j^-(p)}\right).
\end{align*}
Consider the following composite maps for $1 \leq s \leq p-1$ (note that the blocks appear because $\rho_a$ is the right multiplication by $a$): 
\begin{align*}
&h^+_{s,j,a} : \mathcal{P}^{+}(s) \overset{I^+_{s,j}}{\longrightarrow} \mathcal{P}^+_j(s) \overset{i_{\mathcal{P}^+_j(s)}}{\longrightarrow} Q(s) \overset{\rho_a}{\longrightarrow} Q(s) \overset{p_{\mathcal{P}_j^+(s)}}{\longrightarrow} \mathcal{P}^+_j(s)  \overset{\left(I^+_{s,j}\right)^{-1}}{\longrightarrow} \mathcal{P}^+(s),\\
&h^-_{s,j,a} : \mathcal{P}^{-}(s) \overset{I^-_{s,j}}{\longrightarrow} \mathcal{P}^-_j(s) \overset{i_{\mathcal{P}^-_j(s)}}{\longrightarrow} Q(p-s) \overset{\rho_a}{\longrightarrow} Q(p-s) \overset{p_{\mathcal{P}_j^-(s)}}{\longrightarrow} \mathcal{P}^-_j(s)  \overset{\left(I^-_{s,j}\right)^{-1}}{\longrightarrow} \mathcal{P}^-(s),
\end{align*}
where $I_{s,j}^+$ and $I^-_{s,j}$ are the isomorphisms defined by (see Proposition \ref{baseBloc}):
\begin{align*}
&I_{s,j}^+(b^+_i(s)) = B^{++}_{ij}(s),\: I_{s,j}^+(x^+_i(s))=X^{-+}_{ij}(s),\: I_{s,j}^+(y^+_i(s))=Y^{-+}_{ij}(s),\: I_{s,j}^+(a^+_i(s)) = A^{++}_{ij}(s),\\
&I_{s,j}^-(b^-_i(s)) = B^{--}_{ij}(p-s),\: I_{s,j}^-(x^-_i(s))=X^{+-}_{ij}(p-s),\: I_{s,j}^-(y^-_i(s))=Y^{+-}_{ij}(p-s),\\
& I_{s,j}^-(a^-_i(s))=A^{--}_{ij}(p-s).
\end{align*}
For $s = p$, consider:
\begin{align*}
&h^+_{p,j,a} : \mathcal{X}^{+}(p) \overset{I^+_{p,j}}{\longrightarrow} \mathcal{X}^+_j(p) \overset{i_{\mathcal{X}^+_j(p)}}{\longrightarrow} Q(p) \overset{\rho_a}{\longrightarrow} Q(p) \overset{p_{\mathcal{X}_j^+(p)}}{\longrightarrow} \mathcal{X}^+_j(p)  \overset{\left(I^+_{p,j}\right)^{-1}}{\longrightarrow} \mathcal{X}^+(p),\\
&h^-_{p,j,a} : \mathcal{X}^{-}(p) \overset{I^-_{p,j}}{\longrightarrow} \mathcal{X}^-_j(p) \overset{i_{\mathcal{X}^-_j(p)}}{\longrightarrow} Q(0) \overset{\rho_a}{\longrightarrow} Q(0) \overset{p_{\mathcal{X}_j^-(p)}}{\longrightarrow} \mathcal{X}^-_j(p)  \overset{\left(I^-_{p,j}\right)^{-1}}{\longrightarrow} \mathcal{X}^-(p)
\end{align*}
where $I^+_{p,j}$ and $I^-_{p,j}$ are the isomorphisms defined by (see Proposition \ref{baseBloc}):
$$ I^+_{p,j}(v^+_i(p)) = A^{++}_{ij}(p) \:\:\: \text{ and } \:\:\:  I^-_{p,j}(v^-_i(p)) = A^{--}_{ij}(0). $$
Then for $1 \leq s \leq p-1$:
$$t_{\mathcal{P}_j^{\alpha}(s)}\!\left(p_{\mathcal{P}_j^{\alpha}(s)} \circ \rho_a \circ i_{\mathcal{P}_j^{\alpha}(s)}\right) = t_{\mathcal{P}^{\alpha}(s)}\!\left(h^{\alpha}_{s,j,a}\right)$$
and for $s = p$:
$$t_{\mathcal{X}_j^{\alpha}(p)}\!\left(p_{\mathcal{X}_j^{\alpha}(p)} \circ \rho_a \circ i_{\mathcal{X}_j^{\alpha}(p)} \right) = t_{\mathcal{X}^{\alpha}(p)}\!\left(h^{\alpha}_{p,j,a}\right).$$

\noindent We must determine the endomorphism $h^{\alpha}_{s,j,a}$ when $a$ is replaced by the elements given in Proposition \ref{propArike}. Using (\ref{actionWs}), we get:
$$ \forall\, s' \neq s, \forall \, j, \:\: h^{\pm}_{s',j,w^+_s} = 0 \:\:\:\text{ and } \:\:\: h^{-}_{s,j,w^+_s} = 0 $$
and:
$$\forall\, j, \:\: h^{+}_{s,j,w^+_s} = p^+_s.$$
Since this does not depend on $j$ and since the block $Q(s)$ contains $s$ copies of $\mathcal{P}^+(s)$, we find that $t_{\bar U_q}(\rho_{w^+_s}) = st_{\mathcal{P}^+(s)}(p^+_s)$. So by Proposition \ref{propArike}, the coefficient of $G_s$ is $t_{\mathcal{P}^+(s)}(p^+_s)$.

Next, assume that $1 \leq s \leq p-1$, and let us compute $h^{\alpha}_{s',j,\Phi^+_{s-1}e_s}$. By (\ref{actionEs}), we see that 
$$ \forall\, s' \not\in\{s, p-s\}, \forall \, j, \:\: h^{\pm}_{s',j,\Phi^+_{s-1}e_s} = 0 \:\:\: \text{ and } \:\:\: \forall\, j, \:\: h^-_{s,j,\Phi^+_{s-1}e_s}=0,\, h^+_{p-s,j,\Phi^+_{s-1}e_s}=0. $$
Then, Proposition \ref{baseBloc} together with Lemma \ref{lemmePoids} gives:
$$ \forall \, j, \:\:  h^-_{p-s,j,\Phi^+_{s-1}e_s}=0 \:\:\:\text{ and } \:\:\:  \forall \, 0 \leq j \leq s-2, \:\: h^+_{s,j,\Phi^+_{s-1}e_s}=0  \:\:\: \text{ and } \:\:\: h^+_{s,s-1,\Phi^+_{s-1}e_s}=\text{Id}. $$
It follows that $t_{\bar U_q}\left(\rho_{\Phi^+_{s-1}e_s}\right) = t_{\mathcal{P}^+(s)}(\text{Id})$. So by Proposition \ref{propArike}, the coefficient of $\chi^+_s$ is $t_{\mathcal{P}^+(s)}(\text{Id})$.\\

\noindent We now consider $h^{\pm}_{s',j,\Phi^-_{s-1}e_{p-s}}$. This time, (\ref{actionEs}) shows that
$$ \forall\, s' \not\in\{s, p-s\}, \forall \, j, \:\: h^{\pm}_{s',j,\Phi^-_{s-1}e_{p-s}} = 0 \:\:\: \text{ and } \:\:\: \forall\, j, \:\: h^-_{p-s,j,\Phi^-_{s-1}e_{p-s}}=0,\, h^+_{s,j,\Phi^-_{s-1}e_{p-s}}=0. $$
Then, Proposition \ref{baseBloc} together with Lemma \ref{lemmePoids} gives:
$$ \forall \, j, \:\:  h^+_{p-s,j,\Phi^-_{s-1}e_{p-s}}=0 \:\:\: \text{ and } \:\:\: \forall \, 0 \leq j \leq s-2, \:\: h^-_{s,j,\Phi^-_{s-1}e_{p-s}}=0  \:\:\: \text{ and } \:\:\: h^-_{s,s-1,\Phi^-_{s-1}e_{p-s}}=\text{Id}. $$
It follows that $t_{\bar U_q}\left(\rho_{\Phi^-_{s-1}e_{p-s}}\right) = t_{\mathcal{P}^-(s)}(\text{Id})$. So by Proposition \ref{propArike}, the coefficient of $\chi^-_s$ is $t_{\mathcal{P}^-(s)}(\text{Id})$.\\
Finally, in the case where $s=p$:
$$ \forall\, s' \neq p, \forall\, j, \:\: h^{\pm}_{s', j, \Phi^+_{p-1}e_{p}} = 0 \:\:\: \text{ and } \:\:\: h^-_{p,j,\Phi^+_{p-1}e_p} = 0.$$
Then, Proposition \ref{baseBloc} together with Lemma \ref{lemmePoids} gives:
$$ \forall \, 0 \leq j \leq p-2, \:\: h^+_{p,j,\Phi^+_{p-1}e_p}=0  \:\:\:\text{ and }\:\:\: h^+_{p,p-1,\Phi^+_{p-1}e_p}=\text{Id}. $$
It follows that $t_{\bar U_q}\left(\rho_{\Phi^+_{p-1}e_p}\right) = t_{\mathcal{X}^+(p)}(\text{Id})$. So by Proposition \ref{propArike}, the coefficient of $\chi^+_p$ is $t_{\mathcal{X}^+(p)}(\text{Id})$
One similarly gets the coefficient of $\chi^-_p$.
\end{proof}

\indent By Proposition \ref{propArike}, the coefficient of $G_s$ is also given by: $\frac{1}{p-s}t_{\bar U_q}(\rho_{w^-_s})$. Taking back the notations of the proof above, we see using (\ref{actionWs}) that
$$ \forall\, s' \neq p-s, \forall \, j, \:\: h^{\pm}_{s',j,w^-_s} = 0 \:\:\: \text{ and } \:\:\: h^{+}_{p-s,j,w^-_s} = 0 $$
and:
$$\forall\, j, \:\: h^{-}_{p-s,j,w^-_s} = p^-_{p-s}.$$
Since this does not depend on $j$ and since the block $Q(s)$ contains $p-s$ copies of $\mathcal{P}^-(p-s)$, we find that $t_{\bar U_q}(\rho_{w^-_s}) = (p-s)t_{\mathcal{P}^-(p-s)}(p^-_{p-s})$. So by Proposition \ref{propArike}, the coefficient of $G_s$ is $t_{\mathcal{P}^-(p-s)}(p^-_{p-s})$. We thus have: 
\begin{equation}\label{tracePplusMoins}
t_{\mathcal{P}^-(p-s)}(p^-_{p-s}) = t_{\mathcal{P}^+(s)}(p^+_{s}).
\end{equation}
Note that there is an elementary way to see this. Indeed, the morphisms $P^+_s$ and $\bar P^-_{p-s}$ defined in (\ref{morphismesP}) satisfy:
$$ \bar P^-_{p-s} \circ P^+_s = p^+_s, \:\:\: P^+_s \circ \bar P^-_{p-s} = p^-_{p-s}. $$
Hence, we recover (\ref{tracePplusMoins}) by property of the traces. From this, we deduce the following corollary.

\begin{corollary}
Let
$$\varphi = \sum_{s=1}^{p} \left(\lambda^{+}_s\chi^{+}_s + \lambda^{-}_s\chi^{-}_s\right) + \sum_{s'=1}^{p-1}\mu_{s'} G_{s'} \in \text{\em SLF}\left(\bar U_q\right).$$
Then the trace $t^{\varphi} = \Pi_{\bar U_q}^{-1}(\varphi)$ associated to $\varphi$ is given by:
$$ t^{\varphi}_{\mathcal{X}^{\pm}(p)}(\text{\em Id}) = \lambda^{\pm}_p,\:\:\: t^{\varphi}_{\mathcal{P}^{\pm}(s)}(\text{\em Id}) = \lambda^{\pm}_s, \:\:\: t^{\varphi}_{\mathcal{P}^+(s')}(p^+_{s'})=t^{\varphi}_{\mathcal{P}^-(p-s')}(p^-_{p-s'})=\mu_{s'}. $$
\end{corollary}

\subsection{Symmetric linear form corresponding to the modified trace on $\Proj_{\bar U_q}$}\label{ModTrace}
\indent Let $H$ be a finite dimensional Hopf algebra. Let us recall that a {\em modified trace} $\mathsf{t}$ on $\Proj_H$ is a trace which satisfies the additional property that for $U \in \Proj_{H}$, for each $H$-module $V$ and for $f \in \End_{H}(U \otimes V)$  we have:
\begin{equation*}
\mathsf{t}_{U \otimes V}(f) = \mathsf{t}_U(\mathrm{tr}_R(f))
\end{equation*}
where $\mathrm{tr}_R = \mathrm{Id} \otimes \mathrm{tr}_q$ is the right partial quantum trace (see \cite[(3.2.2)]{GKP}). These modified traces are actively studied, having for motivation the construction of invariants in low dimensional topology. We refer to \cite{GKP} for the general theory in a categorical framework which encapsulates the case of $\Proj_H$.
\medskip\\
\indent In \cite{BBGe}, it is shown that there exists a unique up to scalar modified trace $\mathsf{t} = (\mathsf{t}_U)$ on $\Proj_{\bar U_q}$. Uniqueness comes from the fact that $\mathcal{X}^+(p)$ is both a simple and a projective module. The values of this trace are given by:
\begin{equation*}
\begin{array}{lll}
\mathsf{t}_{\mathcal{X}^+(p)}(\text{Id}) = (-1)^{p-1}, & \mathsf{t}_{\mathcal{X}^-(p)}(\text{Id}) = 1, & \mathsf{t}_{\mathcal{P}^+(s)}(\text{Id}) = (-1)^s(q^s+q^{-s}),\\
 \mathsf{t}_{\mathcal{P}^-(s)}(\text{Id}) = (-1)^{p-s-1}(q^s+q^{-s}), & \mathsf{t}_{\mathcal{P}^+(s)}(p^+_s) = (-1)^s[s]^2 & \mathsf{t}_{\mathcal{P}^-(s)}(p^-_s) = \mathsf{t}_{\mathcal{P}^+(p-s)}(p^+_{p-s}).
\end{array}
\end{equation*}
\indent Let $H$ be a finite dimensional unimodular pivotal Hopf algebra with pivotal element $g$ and let $\mu \in H^*$ be a right co-integral on $H$, which means that
\begin{equation*}
\forall \, x \in H, \:\: (\mu \otimes \text{Id})(\Delta(x)) = \mu(x)1.
\end{equation*}
From \cite{radford}, we know that $\mu(g \cdot)$ is a symmetric linear form. In the recent paper \cite{BBG}, it is shown that modified traces on $\Proj_H$ are unique up to scalar, and that the corresponding symmetric linear forms are scalar multiples of $\mu(g\cdot)$. Here, we show how Theorem \ref{MainResult} and computations made in \cite{GT} (see also \cite{arike}) and \cite{FGST} quickly allow us to recover this result in the case of $H = \bar U_q$. First, recall that right integrals $\mu_{\zeta}$ of $\bar U_q$ are given by:
\begin{equation*}
\mu_{\zeta}(F^mE^nK^j) = \zeta \delta_{m, p-1}\delta_{n,p-1}\delta_{j,p+1},
\end{equation*}
where $\zeta$ is an arbitrary scalar. Hence:
\begin{equation*}
\mu_{\zeta}(K^{p+1} F^mE^nK^j) = \zeta \delta_{m, p-1}\delta_{n,p-1}\delta_{j,0}.
\end{equation*}
Using formulas given in \cite{GT} (see also \cite{arike}\footnote{In notations of \cite{arike}, we have $e_s = \sum_{t=1}^se^+(s,t) + \sum_{u=1}^{p-s}e^-(p-s,u)$.}), we have ($1 \leq s \leq p-1$):
\begin{align*}
e_0 &= \frac{(-1)^{p-1}}{2p[p-1]!^2}\sum_{t=0}^{p-1}\sum_{l=0}^{2p-1}q^{-(-2t-1)l}F^{p-1}E^{p-1}K^l + (\text{terms of lower degree in } E \text{ and } F),\\
e_s &= \alpha_s\sum_{t=0}^{p-1}\sum_{l=0}^{2p-1}q^{-(s-2t-1)l}F^{p-1}E^{p-1}K^l + (\text{terms of lower degree in } E \text{ and } F),\\
e_p &= \frac{1}{2p[p-1]!^2}\sum_{t=0}^{p-1}\sum_{l=0}^{2p-1}q^{-(p-2t-1)l}F^{p-1}E^{p-1}K^l + (\text{terms of lower degree in } E \text{ and } F),
\end{align*}
where $\alpha_s$ is given in the last page of \cite{arike} as:
$$ \alpha_s = -\frac{(-1)^{p-s-1}}{2p[p-s-1]!^2[s-1]!^2}\left(\sum_{l=1}^{s-1}\frac{1}{[l][s-l]} - \sum_{l=1}^{p-s-1}\frac{1}{[l][p-s-l]}\right). $$
In order to simplify this, it is observed in \cite[Proof of Proposition 2]{murakami}, that
$$ \sum_{l=1}^{s-1}\frac{1}{[l][s-l]} - \sum_{l=1}^{p-s-1}\frac{1}{[l][p-s-l]} = \frac{-(q^s + q^{-s})}{[s]^{2}}. $$
So, since:
$$ [p-s-1]!^2[s-1]!^2 = \frac{[p-1]!^2}{[s]^2}, $$
we get:
\begin{equation*}
\alpha_s = \frac{(-1)^{p-s-1}}{2p[p-1]!^2}(q^s+q^{-s}).
\end{equation*}
Using formulas given in \cite{FGST} (see also \cite[Prop. II.3.19]{ibanez}), we have:
\begin{align*}
w_s^+ &= \frac{(-1)^{p-s-1}}{2p[p-1]!^2}[s]^2sF^{p-1}E^{p-1} + (\text{other monomials}),\\
w_s^- &= \frac{(-1)^{p-s-1}}{2p[p-1]!^2}[s]^2 (p-s)F^{p-1}E^{p-1} + (\text{other monomials}).
\end{align*}
We now use Proposition \ref{propArike} to get the coefficients of $\mu_{\zeta}(K^{p+1}\cdot)$ in the GTA basis. For instance:
\begin{align*}
\frac{\mu_{\zeta}(K^{p+1}w^+_s)}{s} &= \zeta \frac{(-1)^{p-s-1}}{2p[p-1]!^2}[s]^2,\\
\mu_{\zeta}(K^{p+1}\Phi^+_{s-1}e_s) &= \frac{\alpha_s}{2p}\mu_{\zeta}\left(K^{p+1}F^{p-1}E^{p-1}\sum_{t=0}^{p-1}\sum_{l,j=0}^{2p-1} q^{-(s-1)(l+j)+2tl}K^{l+j}\right)\\
& = \zeta\frac{\alpha_s}{2p}\sum_{t=0}^{p-1}\sum_{l=0}^{2p-1} q^{2tl} = \zeta\alpha_s.
\end{align*}
Choose the normalization factor to be $\zeta = (-1)^{p-1}2p[p-1]!^2$, and let $\mu$ be the so-normalized integral. Then:
\begin{align*}
\mu(K^{p+1}\cdot) = (-1)^{p-1}\chi^+_p + \chi^-_p + \sum_{s=1}^{p-1} & \left( (-1)^s(q^s+q^{-s})\chi^+_s + (-1)^{p-s-1}(q^s+q^{-s})\chi^-_s \right. \\
& \left. + \: (-1)^s[s]^2G_s \right).
\end{align*}
By Theorem \ref{MainResult}, we recover $\Pi_{\bar U_q}(\mathsf{t}) = \mu(K^{p+1}\cdot)$.

\section{Multiplication rules in the GTA basis}\label{multiplication}
\indent We mentioned in section \ref{sectionSLF} that $\SLF(\bar U_q)$ is a commutative algebra. In this section, we address the problem of the decomposition in the GTA basis of the product of two elements in this basis. The resulting formulas are surprisingly simple.
\smallskip\\
\indent Let us start by recalling some facts. For every $\bar U_q$-module $V$, we define the character of $V$ as (see (\ref{defT}) for the definition of $T$):
$$ \chi^V = \text{tr}(\overset{V}{T}). $$
This splits on extensions:
\begin{equation*}
0 \rightarrow V \rightarrow M \rightarrow W \rightarrow 0 \:\: \implies \:\: \chi^M = \chi^V + \chi^W.
\end{equation*}
Due to the fact that $\bar U_q$ is finite dimensional, every finite dimensional $\bar U_q$-module has a composition series (\textit{i.e.} is constructed by successive extensions by simple modules). It follows that every $\chi^V$ can be written as a linear combination of the $\chi^{\alpha}_s = \chi^{\mathcal{X}^{\alpha}(s)}$. Moreover, we see by definition of the product on $\bar U_q^*$ that
\begin{equation}\label{Ttens}
\overset{V \otimes W}{T} = \overset{V}{T}_1 \overset{W}{T}_2
\end{equation}
where $\overset{V}{T}_1 = \overset{V}{T} \otimes I_{\dim(W)}$ and $\overset{W}{T}_2 = I_{\dim(V)} \otimes \overset{W}{T}$. Thus $\chi^{V \otimes W} = \chi^V\chi^W$. Hence multiplying two $\chi$'s is equivalent to tensoring two simples modules and finding the decomposition into simple factors. This means that
\begin{equation*}
\vect(\chi^{\alpha}_s)_{\alpha \in \{\pm\}, 1 \leq s \leq p} \overset{\sim}{\rightarrow} \mathfrak{G}(\bar U_q) \otimes_{\mathbb{Z}} \mathbb{C}, \:\:\: \chi^I \mapsto [I]
\end{equation*}
where $\mathfrak{G}(\bar U_q)$ is the Grothendieck ring of $\bar U_q$. By \cite{FGST}, we know the structure of $\mathfrak{G}(\bar U_q)$. Recall the decomposition formulas (with $2 \leq s \leq p-1$):
\begin{equation*}
\mathcal{X}^-(1) \otimes \mathcal{X}^{\alpha}(s) \cong \mathcal{X}^{-\alpha}(s), \:\:\: \mathcal{X}^+(2) \otimes \mathcal{X}^{\alpha}(s) \cong \mathcal{X}^{\alpha}(s-1) \oplus \mathcal{X}^{\alpha}(s+1), \:\:\: \mathcal{X}^+(2) \otimes \mathcal{X}^{\alpha}(p) \cong \mathcal{P}^{\alpha}(p-1) 
\end{equation*}
so that
\begin{equation}\label{produitChi}
\chi^-_1\chi^{\alpha}_s = \chi^{-\alpha}_s, \:\:\: \chi^+_2 \chi^{\alpha}_s = \chi^{\alpha}_{s-1} + \chi^{\alpha}_{s+1}, \:\:\: \chi^{+}_2\chi^{\alpha}_p = 2\chi^{\alpha}_{p-1} + 2\chi^{-\alpha}_1.
\end{equation}
We see in particular that $\chi^+_2$ generates the subalgebra $\vect(\chi^{\alpha}_s)_{\alpha \in \{\pm\}, 1 \leq s \leq p}$. The $\chi^{\alpha}_s$ are expressed as Chebyschev polynomials of $\chi^+_2$, see \cite[section 3.3]{FGST} for details.

\begin{theorem}\label{ProduitArike}
The multiplication rules in the GTA basis are entirely determined by (\ref{produitChi}) and by the following formulas:
\begin{align}
&\chi^+_2G_1 = [2]G_2, \label{ChiG1}\\
&\chi^+_2G_s = \frac{[s-1]}{[s]}G_{s-1} + \frac{[s+1]}{[s]}G_{s+1} \:\text{ for } 2 \leq s \leq p-2,\label{ChiG}\\
&\chi^+_2G_{p-1} = [2]G_{p-2}, \label{ChiGpMoins1}\\
&\chi^-_1G_s = -G_{p-s} \:\text{ for all } s,\label{ChiMoinsGs}\\
&G_sG_t = 0 \:\text{ for all } s, t.\label{GG}
\end{align}
\end{theorem}
Before giving the proof, let us deduce a few consequences.
\begin{corollary}\label{coroVermaIdeal} For all $1 \leq s \leq p-1$ we have:
$$ G_s = \frac{1}{[s]}\chi^+_sG_1, \:\:\:\: \chi^+_pG_1 = 0. $$
It follows that $(\chi^+_s + \chi^-_{p-s})G_t = 0$, and that $\mathcal{V} = \vect(\chi^+_s + \chi^-_{p-s}, \chi^+_p, \chi^-_p)_{1 \leq s \leq p-1}$ is an ideal of $\SLF(\bar U_q)$.
\end{corollary}
\begin{proof}[Proof of Corollary \ref{coroVermaIdeal}]
The formulas for $\chi^+_sG_1$ are proved by induction using $\chi^+_{s+1}= \chi^+_s\chi^+_2 - \chi^+_{s-1}$ together with formula (\ref{ChiG}). We deduce:
$$ (\chi^+_s + \chi^-_{p-s})G_t = \frac{\chi^+_t}{[t]}(\chi^+_sG_1 + \chi^-_{p-s}G_1) = \frac{\chi^+_t}{[t]}([s]G_s + [s]\chi^-_1G_{p-s}) = 0 .$$
It is straightforward that $\mathcal{V}$ is stable by multiplication by $\chi^+_2$, so it is an ideal.
\end{proof}
\begin{remark}
We have $\chi^{\mathcal{P}^{\alpha}(s)} = 2\left(\chi^{\alpha}_s + \chi^{-\alpha}_{p-s}\right)$ for $1 \leq s \leq p-1$. Thus $\mathcal{V}$ is generated by characters of the projective modules. It is well-known that if $H$ is a finite dimensional Hopf algebra, then the full subcategory of finite dimensional projective $H$-modules is a tensor ideal. Thus we can deduce without any computation that $\mathcal{V}$ is stable under the multiplication by every $\chi^{I}$.
\end{remark}

\indent We now proceed with the proof of the theorem. Observe that we cannot apply Proposition \ref{propArike} to show it since we do not know expressions of $\Delta(e_s)$ and $\Delta(w^{\pm}_s)$ which are easy to evaluate in the GTA basis. Recall (\cite{KS}, see also \cite{ibanez}) the following fusion rules:
\begin{align}
&\mathcal{X}^-(1) \otimes \mathcal{P}^{\alpha}(s) \cong \mathcal{P}^{-\alpha}(s)  \:\:\: \text{ for all } s, \label{chi1MoinsTensPs}\\
&\mathcal{X}^+(2) \otimes \mathcal{P}^{\alpha}(1) \cong 2\mathcal{X}^{-\alpha}(p) \oplus \mathcal{P}^{\alpha}(2),\label{chi2TensP1}\\
&\mathcal{X}^+(2) \otimes \mathcal{P}^{\alpha}(s) \cong \mathcal{P}^{\alpha}(s-1) \oplus \mathcal{P}^{\alpha}(s+1) \:\:\: \text{ for } 2 \leq s \leq p-1,\label{chi2TensPs}\\
&\mathcal{X}^+(2) \otimes \mathcal{P}^{\alpha}(p-1) \cong 2\mathcal{X}^{\alpha}(p) \oplus \mathcal{P}^{\alpha}(p-2)\label{chi2TensPpMoins1}.
\end{align}
They imply the following key lemma.
\begin{lemma}\label{LemmeProdChiGs}
There exist scalars $\gamma_s, \beta_s, \lambda_s, \eta_s, \delta_s$ such that
\begin{align*}
&\chi^+_2G_s = \beta_s G_{s-1} + \gamma_s G_{s+1} + \lambda_s\!\left( \chi^+_{s-1} + \chi^-_{p-s+1} - \chi^+_{s+1} - \chi^-_{p-s-1} \right) \:\:\:\: (\text{for } 2 \leq s \leq p-2),\\
&\chi^+_2G_1 = \gamma_1 G_{2} + \lambda_1\!\left(\chi^-_p - \chi^+_2 - \chi^-_{p-2} \right), \:\:\:\: \chi^+_2G_{p-1} = \beta_{p-1} G_{p-2} + \lambda_{p-1}\!\left( \chi^+_{p-2} + \chi^-_{2} - \chi^+_p \right),\\
&\chi^-_1G_s = \eta_s G_{p-s} + \delta_s\!\left( \chi^+_{p-s} + \chi^-_s \right).
\end{align*}
\end{lemma}
\begin{proof}
Let us fix $2 \leq s \leq p-2$; by (\ref{defChis}), (\ref{defGs}), (\ref{Ttens}) and (\ref{chi2TensPs}) we have:
\begin{align*}
\chi^+_2G_s &\in \vect\!\left(\overset{\mathcal{X}^{+}(2)}{T_{ij}}\cdot\overset{\mathcal{P}^{+}(s)}{T_{kl}}, \overset{\mathcal{X}^{+}(2)}{T_{ij}}\cdot\overset{\mathcal{P}^{-}(p-s)}{T_{kl}}\right)_{ijkl} = \vect\!\left(\overset{\mathcal{X}^{+}(2) \otimes \mathcal{P}^+(s)}{T_{ijkl}}, \overset{\mathcal{X}^{+}(2) \otimes \mathcal{P}^-(p-s)}{T_{ijkl}}\right)_{ijkl}\\
&= \vect\!\left(\overset{\mathcal{P}^+(s-1)}{T_{ij}}, \overset{\mathcal{P}^+(s+1)}{T_{ij}}, \overset{\mathcal{P}^-(p-s+1)}{T_{ij}}, \overset{\mathcal{P}^-(p-s-1)}{T_{ij}}\right)_{ij}
\end{align*}
where $\overset{V}{T}\!_{ij}$ is the matrix element at the $i$-th row and $j$-th column of the representation matrix $\overset{V}{T}$ and $\overset{V \otimes W}{T}\!\!\!\!_{ijkl}$ is the matrix element at the $(i,j)$-th row and $(k,l)$-th column of the representation matrix $\overset{V \otimes W}{T}$. Hence, since $\chi^+_2G_s$ is symmetric, it is necessarily of the form 
$$
\chi^+_2G_s = \beta_s G_{s-1} + \gamma_s G_{s+1} + z_1 \chi^+_{s-1} + z_2 \chi^+_{s+1} + z_3 \chi^-_{p-s+1} + z_4 \chi^-_{p-s-1}.
$$
Evaluating this equality on $K$ and $K^2$, we find (since $G_t(K^{l}) = 0$ for all $t$ and $l$):
\[ [s-1](z_1 - z_3) + [s+1](z_2 - z_4) = 0, \:\:\:\: [s-1]_{q^2}(z_1 - z_3) + [s+1]_{q^2}(z_2 - z_4) = 0, \]
with $[n]_{q^2} = \frac{q^{2n} - q^{-2n}}{q^2 - q^{-2}}$. The determinant of this linear system with unknowns $z_1 - z_3, z_2 - z_4$ is $\frac{2\sin( (s-1)\pi/p ) \sin( (s+1)\pi/p )}{\sin( \pi/p )\sin( 2\pi/p )}\left( \cos( (s+1)\pi/p ) - \cos( (s-1)\pi/p ) \right) \neq 0$.  Hence $z_1 = z_3$, $z_2 = z_4$. Moreover, evaluating the above equality on $1$, we find $p(z_1 + z_2) = 0$. Letting $\lambda_s = z_1$, the result follows. The other formulas are obtained in a similar way using (\ref{chi1MoinsTensPs}), (\ref{chi2TensP1}) and (\ref{chi2TensPpMoins1}).
\end{proof}

\indent We will use the Casimir element $C$ of $\bar U_q$ to make computations easier. It is defined by:
\[ C = FE + \frac{qK + q^{-1}K^{-1}}{(q - q^{-1})^2} = \sum_{j=0}^p c_j e_j + \sum_{k = 1}^{p-1} (w^+_k + w^-_k) \in \mathcal{Z}(\bar U_q)\]
where $c_j = \frac{q^j + q^{-j}}{ (q - q^{-1})^2 }$. The second equality is obtained by considering the action of $C$ on the PIMs $\mathcal{P}^{\alpha}(s)$. Observe that
\begin{equation}\label{precCasimir}
\forall \, x \in \bar U_q, \:\:\: \chi^{\alpha}_s(Cx) = \alpha c_s \chi^{\alpha}_s(x), \:\:\: G_s(Cx) = c_sG_s(x) + (\chi^+_s + \chi^-_{p-s})(x). 
\end{equation}
Then by induction we get $G_s(C^n) = npc_s^{n-1}$ for $n \geq 1$. We will also denote $c_K = \frac{qK + q^{-1}K^{-1}}{(q - q^{-1})^2}$.

\begin{proof}[Proof of Theorem \ref{ProduitArike}]
~\\
\indent \textbullet \hspace{1pt} \textit{Formula \eqref{ChiG}.} \hspace{2pt} We first evaluate the corresponding formula of Lemma \ref{LemmeProdChiGs} on $FE$. It holds $G_t(FE) = G_t(C) = p$, $(\chi^+_t + \chi^-_{p-t})(FE) = (\chi^+_t + \chi^-_{p-t})(C) = p c_t$ for all $t$ and $\chi^+_2G_s(FE) = \chi^+_2(K^{-1})G_s(FE) = [2]p$. Thus we get:
\begin{equation}\label{eqLin1}
\beta_s + \gamma_s  + (c_{s-1} - c_{s+1}) \lambda_s = \beta_s + \gamma_s  - [s] \lambda_s = [2].
\end{equation}
Next, we evaluate the formula of Lemma \ref{LemmeProdChiGs} on $(FE)^2$. On the one hand,
\begin{align*}
(\chi^+_2 G_s)\!\left( (FE)^2 \right) &= \chi^+_2(K^{-2}) G_s\!\left( (FE)^2 \right) = \chi^+_2(K^{-2}) G_s\!\left( C^2 - 2 C c_K + c_K^2 \right)\\
& = \chi^+_2(K^{-2}) G_s(C^2) = 2p(q^2 + q^{-2})c_s.
\end{align*}
For the first equality, we used that $\varphi(E^iF^jK^{l}) = \delta_{i,j}\varphi(E^iF^iK^{l})$ for all $\varphi \in \mathrm{SLF}(\bar U_q)$, that $G_s(K^{l}) = 0$ and that $G_s(FEK^{l}) = 0$ for $1 \leq l \leq p-1$. The third equality is due to \eqref{precCasimir} and to the fact that $(\chi^+_s + \chi^-_{p-s})(K^{l}) = 0$ for $1 \leq l \leq p-1$. 
On the other hand, using again the Casimir element,
\begin{align*}
&\beta_s G_{s-1}\!\left( (FE)^2 \right) + \gamma_s G_{s+1}\!\left( (FE)^2 \right) + \lambda_s\!\left( \chi^+_{s-1} + \chi^-_{p-s+1} - \chi^+_{s+1} - \chi^-_{p-s-1} \right)\!\left( (FE)^2 \right)\\
&= \beta_s G_{s-1}\!\left( C^2 \right) + \gamma_s G_{s+1}\!\left( C^2 \right) + \lambda_s\!\left( \chi^+_{s-1} + \chi^-_{p-s+1} - \chi^+_{s+1} - \chi^-_{p-s-1} \right)\!\left( C^2 \right)\\
&= 2pc_{s-1} \beta_s + 2pc_{s+1} \gamma_s + p(c_{s-1}^2 - c_{s+1}^2) \lambda_s.
\end{align*}
Since $c_{s-1}^2 - c_{s+1}^2 = -(q + q^{-1})c_s[s]$, we get
\begin{equation}\label{eqLin2}
2c_{s-1} \beta_s + 2c_{s+1} \gamma_s -(q + q^{-1})c_s[s]\lambda_s = 2(q^2 + q^{-2})c_s.
\end{equation}
In order to get a third linear equation between $\beta_s$, $\gamma_s$ and $\lambda_s$, we use evaluation on $E^{p-1}F^{p-1}$. This has the advantage to annihilate all the $\chi^{\alpha}_t$ appearing in the formula of Lemma \ref{LemmeProdChiGs}. First:
\begin{equation}\label{actionP}
\begin{split}
E^{p-1}F^{p-1}b_0^{\alpha}(s) &= E^{p-1}y^{\alpha}_{p-s-1}(s) = (-\alpha)^{p-s-1}[p-s-1]!^2E^sy_0^{\alpha}(s)\\
&= (-\alpha)^{p-s-1}\alpha^{s-1}[p-s-1]!^2[s-1]!^2a_0^{\alpha}(s) \\
&= (-\alpha)^{p-s-1}\alpha^{s-1}\frac{[p-1]!^2}{[s]^2}a_0^{\alpha}(s)
\end{split}
\end{equation}
and $E^{p-1}F^{p-1}$ annihilates all the other basis vectors. Hence:
$$
G_s(E^{p-1}F^{p-1}) = 2(-1)^{p-s-1}\frac{[p-1]!^2}{[s]^2}.
$$
Next by (\ref{coproduitMonome}), we have:
$$
\chi^+_2 \otimes \text{Id}\left(\Delta(E^{p-1}F^{p-1})\right) = -[2]E^{p-1}F^{p-1} - q^2E^{p-2}F^{p-2}K.
$$
As in (\ref{actionP}), we find:
\begin{align*}
E^{p-2}F^{p-2}Kb_0^{\alpha}(s) &= (-\alpha)^{p-s}\alpha^sq^{s-1}\frac{[p-1]!^2}{[s+1][s]^2}a_0^{\alpha}(s),\\
E^{p-2}F^{p-2}Kb_1^{\alpha}(s) &= (-\alpha)^{p-s-1}\alpha^{s-1}q^{s-3}\frac{[p-1]!^2}{[s-1][s]^2} a_1^{\alpha}(s)
\end{align*}
and all the others basis vectors are annihilated. Hence:
$$
G_s(E^{p-2}F^{p-2}K) = 2(-1)^{p-s-1}\frac{[p-1]!^2}{[s]^2}\frac{q^{-2}[2]}{[s-1][s+1]}.
$$
We obtain:
$$ \chi^+_2 \otimes G_s\left(\Delta(E^{p-1}F^{p-1})\right) = 2(-1)^{p-s}[p-1]!^2\frac{[2]}{[s-1][s+1]} $$
and thus:
\begin{equation}\label{eqLin3}
\frac{\beta_s}{[s-1]^2} +  \frac{\gamma_s}{[s+1]^2} = \frac{[2]}{[s-1][s+1]}.
\end{equation}
As a result, we have a linear system \eqref{eqLin1}--\eqref{eqLin2}--\eqref{eqLin3} between $\beta_s$, $\gamma_s$ and $\lambda_s$. It is easy to check that $\beta_s = \frac{[s-1]}{[s]}, \gamma_s = \frac{[s+1]}{[s]}, \lambda_s = 0$ is a solution. Moreover this solution is unique. Indeed, a straightforward computation reveals that
\[
\det
\left(
\begin{array}{ccc}
1 & 1 & -[s]\\
2 c_{s-1} & 2c_{s+1} & -(q + q^{-1})c_s[s]\\
\frac{1}{[s-1]^2} & \frac{1}{[s+1]^2} & 0
\end{array}
\right)
= \frac{[s]^2}{[s-1]^2} + \frac{[s]^2}{[s+1]^2} > 0.
\]

\indent \textbullet \hspace{1pt} \textit{Formulas \eqref{ChiG1} and \eqref{ChiMoinsGs}.} \hspace{2pt} Evaluating as above the corresponding formulas of Lemma \ref{LemmeProdChiGs} on $FE$ and $(FE)^2$, one gets linear systems with non-zero determinants.  It is then easy to see that $\beta_1 = [2], \lambda_1 = 0$ and $\eta_s = -1, \delta_s = 0$ are the unique solutions of each of these two systems.
\medskip\\
\indent \textbullet \hspace{1pt} \textit{Formula \eqref{ChiGpMoins1}.} \hspace{2pt} It can be deduced from the formulas already shown:
\[ \chi^+_2 G_{p-1} = -\chi^+_2 \chi^-_1 G_1 = -\chi^-_1 [2] G_2 = [2] G_{p-2}. \]
\indent \textbullet \hspace{1pt} \textit{Formula \eqref{GG}.} \hspace{2pt} Recall the isomorphism of algebras $\mathcal{D}$ defined in (\ref{morphismeDrinfeld}). Taking into account that $\varphi(K^iF^mE^n)=0$ if $n \neq m$ for any $\varphi \in \SLF(\bar U_q)$ and that $G_s(K^i) = 0$ for all $i$, and making use of the expression of $RR'$ given in \cite{FGST}, we get:
\begin{align*}
\mathcal{D}(G_s) &= \sum_{n=0}^{p-1}\sum_{j=0}^{2p-1}\left(\sum_{i=0}^{2p-1} \frac{(q-q^{-1})^n}{[n]!^2}q^{n(j-i-1)-ij}G_s(K^{p+i+1}E^nF^n)\right)K^jF^nE^n\\
&= \sum_{n=1}^{p-1}\sum_{j=0}^{2p-1}\lambda_{j,n}K^jF^nE^n
\end{align*}
for some coefficients $\lambda_{j,n}$ (observe that $n \geq 1$). From this it follows that for all $\alpha \in \{\pm\}$ and $1 \leq r \leq p-1$: $ \mathcal{D}(G_s)b_0^{\alpha}(r) \in \mathbb{C}a_0^{\alpha}(r)$. By \eqref{actionEs}, we deduce that $\mathcal{D}(G_s) \in \vect(w^{\pm}_r)_{1 \leq r \leq p-1}$ for all $s$. Thus $\mathcal{D}(G_sG_t) = 0$, thanks to \eqref{produitCentre}.
\end{proof}

\vspace{3em}
\hfill\begin{minipage}{.60\linewidth}
\noindent IMAG, Univ Montpellier, CNRS, Montpellier, France.
\\E-mail address: matthieu.faitg@umontpellier.fr
\end{minipage}

\end{document}